\documentclass[a4paper,UKenglish]{eurocg25}
\usepackage{physics}
\usepackage{mathtools}
\usepackage{bm}
\usepackage{todonotes}

\usepackage[capitalise, noabbrev]{cleveref}
\usepackage{thm-restate}

\makeatletter
\newcommand{\proofsubparagraph}{\@startsection{subparagraph}{5}{\z@}%
                                       {3.25ex \@plus1ex \@minus .2ex}%
                                       {-1em}%
           {\color{lipicsGray}\sffamily\normalsize\bfseries}}
\makeatother

\newtheorem{question}[theorem]{Question}
\newtheorem{observation}[theorem]{Observation}

\crefname{observation}{Observation}{Observations}
\Crefname{observation}{Observation}{Observations}
\crefname{question}{Question}{Questions}
\Crefname{question}{Question}{Questions}

\DeclareGraphicsExtensions{.pdf,.png,.eps}

\newcommand{\N}{\mathbb{N}}

\newcommand{\calC}{\mathcal{C}}
\newcommand{\mul}{\,}
\DeclarePairedDelimiterX\set[1]\lbrace\rbrace{#1}

\newcommand{\CIIIvO}{\raisebox{-1.5pt}{\includegraphics[page=2]{figures/small-icons.pdf}}}
\newcommand{\CIVvO}{\raisebox{-1.5pt}{\includegraphics[page=3]{figures/small-icons.pdf}}}
\newcommand{\CIVvI}{\raisebox{-2pt}{\includegraphics[scale=0.9,page=4]{figures/small-icons.pdf}}}
\newcommand{\CVvI}{\raisebox{-2pt}{\includegraphics[page=6]{figures/small-icons.pdf}}}
\newcommand{\CVvO}{\raisebox{-3pt}{\includegraphics[page=7]{figures/small-icons.pdf}}}
\newcommand{\Clarge}{\raisebox{-1.5pt}{\includegraphics[page=8]{figures/small-icons.pdf}}}
\newcommand{\CVIvII}{\raisebox{-1.5pt}{\includegraphics[page=16]{figures/small-icons.pdf}}}
\newcommand{\CVvII}{\raisebox{-1.5pt}{\includegraphics[page=17]{figures/small-icons.pdf}}}
\newcommand{\YTriQuad}{\raisebox{-3pt}{\includegraphics[page=9]{figures/small-icons.pdf}}}
\newcommand{\YQuintQuadTri}{\raisebox{-3pt}{\includegraphics[page=10]{figures/small-icons.pdf}}}
\newcommand{\YTriQuadTri}{\raisebox{-3pt}{\includegraphics[page=12]{figures/small-icons.pdf}}}
\newcommand{\YTriTri}{\raisebox{-3pt}{\includegraphics[page=13]{figures/small-icons.pdf}}}
\newcommand{\YTriTriTri}{\raisebox{-3pt}{\includegraphics[page=14]{figures/small-icons.pdf}}}
\newcommand{\YStar}{\raisebox{-4.5pt}{\includegraphics[page=15]{figures/small-icons.pdf}}}
\newcommand{\YTriQuadQuad}{\raisebox{-3pt}{\includegraphics[page=18]{figures/small-icons.pdf}}}

\newcommand{\cor}[2]{%
\left(#1\leftrightarrow#2\right)
}
\newcommand{\coeff}[1]{%
#1}

\newcommand{\Toth}{T{\'o}th}%

\author[1]{Miriam Goetze\footnote{funded by the Deutsche Forschungsgemeinschaft (DFG, German Research Foundation) -- 520723789}}
\author[2]{Michael Hoffmann}
\author[3]{Ignaz Rutter\footnote{funded by the Deutsche Forschungsgemeinschaft (DFG, German Research Foundation) -- 541433306}}
\author[1]{Torsten~Ueckerdt}
\affil[1]{Karlsruhe Institute of Technology, Germany\\
\texttt{miriam.goetze@kit.edu}, \texttt{torsten.ueckerdt@kit.edu}}
\affil[2]{Department of Computer Science, ETH Z\"urich, Switzerland\\
\texttt{hoffmann@inf.ethz.ch}}
\affil[3]{University of Passau, Germany\\
\texttt{rutter@fim.uni-passau.de}}

\authorrunning{M. Goetze, M. Hoffmann, I. Rutter, T. Ueckerdt}
\ArticleNo{56}

\title{Crossing Number of 3-Plane Drawings\footnote{This research was initiated at the Workshop on \emph{Graph and Network Visualization} (GNV~2024) in Heiligkreuztal, Germany, June 23--28, 2024.}}

\begin{document}

\maketitle

\begin{abstract}
    We study $3$-plane drawings, that is, drawings of graphs in which every edge has at most three crossings.
    We show how the recently developed Density Formula for topological drawings of graphs~\cite{kaufmann2023density} can be used to count the crossings in terms of the number~$n$ of vertices.
    As a main result, we show that every~$3$-plane drawing has at most~$5.5(n-2)$ crossings, which is tight.
    In particular, it follows that every~$3$-planar graph on~$n$ vertices has crossing number at most~$5.5n$, which improves upon a recent bound~\cite{bekos2024k-planar} of~$6.6n$. 
    To apply the Density Formula, we carefully analyze the interplay between certain configurations of cells in a~$3$-plane drawing.
    As a by-product, we also obtain an alternative proof for the known statement that every~$3$-planar graph has at most~$5.5(n-2)$ edges.
\end{abstract}

\section{Introduction}

One of the most basic combinatorial questions one can ask for a class of graphs is: How many edges can a graph from this class have as a function of the number~$n$ of vertices? Prominent examples include upper bounds of~$\binom{n}{2}$ for the class of all graphs and~$\frac{n^2}{4}$ for bipartite graphs. These bounds are immediate consequences of the definition of these graph classes, and they are tight, that is, there exist graphs in the class with exactly this many edges. But for several other graph classes good upper bounds on the number of edges are much more challenging to obtain. Notably this holds for classes that relate to the existence of certain geometric representations. One the most fundamental questions one can ask about a class of geometrically represented graphs is: What is the minimum number of edge crossings required in such a representation, as a function of the number~$n$ of vertices? 
We study both of these fundamental questions in combination, for the class of $3$-planar graphs. A graph is \emph{$k$-planar} if it can be drawn in the plane such that every edge has at most~$k$ crossings. 
The study of~$k$-planar graphs goes back to Ringel~\cite{Ringel65} and has been a major focus in graph drawing over the past two decades~\cite{dlm-sgdbp-20}, as a natural generalization of planar graphs~($k=0$). 

The maximum number of edges in a simple~$k$-planar graph on~$n$ vertices is known to be at most~$c_k(n-2)$, where~$c_0=3$, $c_1=4$~\cite{bodendiek1983bemerkungen}, $c_2=5$~\cite{pt-gdfcpe-96,pt-gdfcpe-97}, $c_3=5.5$~\cite{prtt-iclfmcsg-04,prtt-iclfmcsg-06}, $c_4=6$~\cite{ackerman2019topological}, and~$c_k\le 3.81\sqrt{k}$, for general~$k\ge 5$~\cite{ackerman2019topological}. The bounds for~$k\le 2$ are tight and those for~$k\le 4$ are tight up to an additive constant~\cite{ackerman2019topological,bekos2017optimal}. The bounds for~$k\le 4$ also generalize to \emph{non-homotopic} drawings of multigraphs~\cite{ptt-cbne-20,ptt-cbne-22}, that is, where every continuous transformation that transforms one copy of an edge to another passes over a vertex. Interestingly, the upper bound for~$3$-planar graphs is tight in this more general setting only~\cite{bekos2017optimal,bungener2024improving}.

The \emph{crossing number} of a drawing~$\Gamma$ is the number of edge crossings in~$\Gamma$. The \emph{crossing number}~$\mathrm{cr}(G)$ of a graph~$G$ is the minimum crossing number over all drawings of~$G$. By definition every~$k$-planar graph~$G$ admits a~$k$-plane drawing and thus
\begin{equation}\label{eq:simple}\tag{S}
\mathrm{cr}(G)\le\frac{km}{2}, 
\end{equation}
where~$m$ denotes the number of edges in~$G$. For a~$k$-planar graph, this simple inequality connects upper bounds on the number of edges with lower bounds on the crossing number. Both of these come together in the well-known Crossing Lemma~\cite[Chapter~45]{az-pb-18}, as the best constants in the Crossing Lemma are obtained by analyzing~$k$-plane drawings~\cite{ackerman2019topological,bungener2024improving,prtt-iclfmcsg-04,prtt-iclfmcsg-06}. 
Conversely, combining the lower bound on~$\mathrm{cr}(G)$ from the Crossing Lemma with an upper bound on~$\mathrm{cr}(G)$ we obtain an upper bound on the number of edges in~$G$. While \eqref{eq:simple} would work here, it is probably not an ideal choice because the graphs for which \eqref{eq:simple} is tight might be very different from those graphs that have a maximum number of edges, for any fixed~$n$. For instance, for a~$1$-planar graph~$G$ we have~$\mathrm{cr}(G)\le n-2$~\cite[Proposition~4.4]{suzuki2020beyondPlanar}, which beats the bound we get by plugging~$m\le 4n-8$ into~\eqref{eq:simple} by a factor of two. Can we obtain similar improvements by bounding~$\mathrm{cr}(G)$ in terms of~$n$, rather than~$m$, for~$k\ge 2$?

Indeed, very recently it has been shown that~$\mathrm{cr}(G)\le 3.\overline{3}n$ if~$G$ is~$2$-planar and~$\mathrm{cr}(G)\le 6.6n$ if~$G$ is~$3$-planar~\cite{bekos2024k-planar}. 
There is some indication that the bound for~$2$-planar graphs could be tight up to an additive constant, as it is achieved by the standard drawings of optimal~$2$-planar graphs (\cref{fig:tightness_2}). But the crossing number of these graphs is not known. 

\begin{figure}[htbp]
        \centering
        \includegraphics{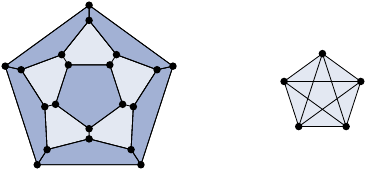}
        \caption{Construction by Pach and \Toth{}~\cite[Figure~3]{pt-gdfcpe-97}.
        Left: A planar drawing with pentagonal faces. Right: To each pentagonal face all diagonals are added.}
        \label{fig:tightness_2}
\end{figure}

In contrast, there exists a family of simple~$3$-planar graphs with~$5.5n-15$ edges whose standard drawings have~$5.5n-21$ crossings (\cref{fig:tightness_3}).
Thus, there is a gap of~$1.1n$ between the lower and the upper bound for the crossing number of~$3$-plane drawings.

\begin{figure}[htbp]
    \centering
    \includegraphics[page=2]{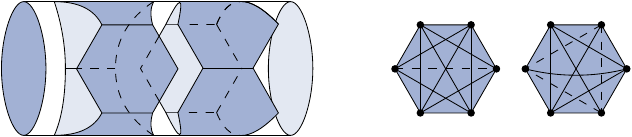}
    \caption{Construction from \cite[Figure~8]{prtt-iclfmcsg-06}. Left: A cylinder with two layers, each consisting of three hexagonal faces. Right: To each face of a layer all but one diagonal is added. To the top and bottom face six diagonals are added. 
    Missing diagonals are represented by dashed lines.  }
    \label{fig:tightness_3}
\end{figure}

\subparagraph*{Results.}

We close the gap and present an upper bound on the crossing number 
of~$3$-plane drawings that is tight up to an additive constant.
Using the same approach we also obtain an alternative proof to show that a~$3$-planar $n$-vertex graph has at most~$5.5(n-2)$ edges.

\begin{restatable}{theorem}{main}
    \label{thm:3-planar_upper_bound_crossings_and_edges}
    Every non-homotopic $3$-plane drawing of a graph on $n$~vertices, $n \geq 3$, contains at most $5.5(n-2)$ edges and at most $5.5(n-2)$ crossings.
\end{restatable}


Our proof relies on the recently developed Density Formula (cf.~\cref{thm:density_formula} below) for topological drawings of graphs~\cite{kaufmann2023density}. 
It relates the number of vertices, edges, and cells of various sizes in a drawing, in a way similar to the Euler Formula in the case of plane graphs.
Previously, the Density Formula has been used to derive upper bounds on the number of edges in~$k$-plane drawings, for~$k\le 2$~\cite{kaufmann2023density}.
In order to apply it to~$3$-plane drawings, to bound the number of crossings, and to obtain tight bounds, we study cells not only in isolation but also as part of what we call \emph{configurations}, which consist of several connected cells.
We then develop a number of new constraints that relate the number of cells and/or configurations of a certain type in any~$3$-plane drawing.
The combination of all these constraints with the Density Formula yields a linear program that we can solve in two different ways---maximizing either the number of edges or the number of crossings---to prove \cref{thm:3-planar_upper_bound_crossings_and_edges}.

Using \cref{thm:3-planar_upper_bound_crossings_and_edges} we can derive better upper bounds on the number of edges in~$k$-planar graphs without short cycles. Plugging our bound of at most~$5.5n$ crossings into the proofs from~\cite{bekos2024k-planar} we obtain that
\begin{itemize}
\item $C_3$-free $3$-planar graphs on~$n$ vertices have at most~$\sqrt[3]{891/8}n<4.812n$ edges (down from~$\approx 5.113n$~\cite[Theorem~18]{bekos2024k-planar}),
\item $C_4$-free $3$-planar graphs on~$n$ vertices have at most~$\sqrt[3]{1'254'825/12'544}n<4.643n$ edges (down from~$\approx 4.933n$~\cite[Theorem~20]{bekos2024k-planar}), and
\item $3$-planar graphs of girth~$5$ on~$n$ vertices have at most~$\sqrt[3]{122,793/1600}n<4.25n$ edges (down from~$\approx 4.516n$~\cite[Theorem~21]{bekos2024k-planar}).
\end{itemize}

\section{Preliminaries}

We consider \emph{drawings} of graphs on the sphere with vertices as points, edges as Jordan arcs, and the usual assumption that any two edges share only finitely many points, each being a common endpoint or a proper crossing, and that no three edges cross in the same point.
We also assume that no edge crosses itself and that no two adjacent edges cross.
As is customary, we do not distinguish between the points and curves in~$\Gamma$ and the vertices and edges of~$G$ they represent, respectively.
The graphs we consider may contain parallel edges, but no loops.
In order to avoid an arbitrary number of parallel edges within a small corridor, a drawing~$\Gamma$ is called \emph{non-homotopic} if every region that is bounded by exactly two parts of edges, called a \emph{lens}, contains a crossing or a vertex in its interior; see \cref{fig:lens}. 

\begin{figure}[htbp]
    \centering
    \includegraphics{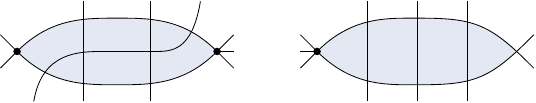}
    \caption{\emph{Left:} A lens (blue) with two crossings in its interior.
    \emph{Right:} An empty lens (blue).}
    \label{fig:lens}
\end{figure}

Let~$\Gamma$ be a drawing of a graph~$G=(V,E)$.
If every edge is crossed at most three times, we say that~$\Gamma$ is \emph{$3$-plane}.
We denote the set of crossings by $X$. 
For~$i \in \set{0,1,2,3}$, let $E_i \subseteq E$ be the set of all edges with exactly $i$ crossings, and let~$E_{\times} = E_1 \cup E_2 \cup E_3$.

\subparagraph{Edge-Segments and Cells.}

An edge with $i$ crossings is split into $i+1$ parts, called \emph{edge-segments}.
An edge-segment is \emph{inner} if both its endpoints are crossings, and \emph{outer} otherwise.
The \emph{planarization} of~$\Gamma$ is the graph obtained by replacing every crossing~$x$ with a vertex of degree~$4$ that is incident to the four edge-segments of~$x$.
We say that the drawing~$\Gamma$ is \emph{connected}, if its planarization is a connected graph, and shall henceforth only consider connected drawings.
Removing all edges and vertices of~$\Gamma$ splits the sphere into several components, called \emph{cells}. 
We denote the set of all cells by~$\calC$.
Since~$\Gamma$ is connected, the \emph{boundary} $\partial c$ of a cell~$c$ corresponds to a cyclic sequence alternating between edge-segments and elements in $V \cup X$ (i.e., vertices and crossings).
If a crossing or a vertex appears multiple times on the boundary of the same cell~$c$, then~$c$ is \emph{degenerate}.
The \emph{size} of a cell~$c$, denoted by $\norm{c}$, is the number of vertex incidences plus the number of edge-segment incidences of~$c$.
Note that incidences with crossings are not taken into account, see \cref{fig:cell-characterization} for examples.
For~$a \in \N$, we denote by $\calC_a = \{ c \in \calC \colon \norm{c} = a\}$ the set of all cells of size~$a$.
\begin{figure}
    \centering
    \includegraphics[page=1]{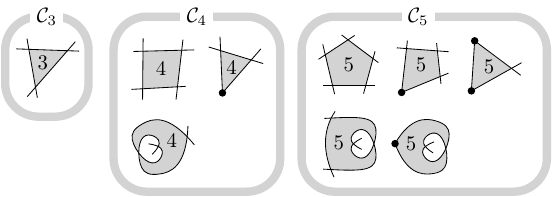}
    \caption{Taken from \cite[Figure~2]{kaufmann2023density}. All types of cells~$c$ of size~$\norm{c} \leq 5$ in a non-homotopic connected drawing on at least three vertices. The bottom row shows the degenerate cells.}
    \label{fig:cell-characterization}
\end{figure}

\begin{theorem}[Density Formula \cite{kaufmann2023density}]
\label{thm:density_formula}
    If $\Gamma$ is a connected drawing with at least one edge, and $t$ is a real number, then
    \[
    \abs{E} = t(\abs{V}-2) - \sum_{c \in \calC}\left( \frac{t-1}{4}\norm{c} - t \right) - \abs{X}
    \]
\end{theorem}

To apply the Density Formula, we count the cells of different sizes.
We distinguish several types of cells based on their size and boundary and denote these by small pictograms, such as $\CIVvI$ or $\CVvI$.
We call a cell \emph{large} if it has size at least~$6$ and write $\Clarge$ for this type of cells.
By abuse of notation, we denote the number of cells of a certain type by their pictogram.

\subparagraph{Configurations}
are connected labeled embedded subgraphs of the planarization of a drawing~$\Gamma$. 
We denote configuration types by pictograms such as $\YTriTriTri$ and $\YTriTri$ (see \cref{fig:configurations_trails}).

\begin{figure}[htbp]
    \centering
    \includegraphics[page=2]{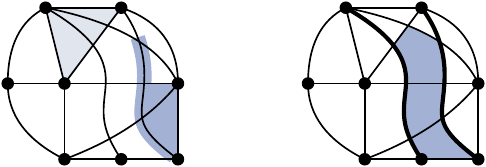}
    \caption{Left: A $\YTriTriTri$-configuration (light blue) and a $\YTriTri$-configuration (dark blue). Right: A $\CVvO$-$\Clarge$-trail (dark blue) and its bounding edges (thick).}
    \label{fig:configurations_trails}
\end{figure}

A configuration is an \emph{$A$-$B$-trail} if its dual is a path~$P$ whose endpoints are cells of type~$A\ne\CIVvO$ and~$B\ne\CIVvO$, respectively, whose edges correspond to inner segments, and whose interior vertices are \CIVvO-cells whose two edge-segments on~$P$ are opposite along their boundary, see \cref{fig:configurations_trails}.
We denote by $\cor{A}{B}$ the number of $A$-$B$-trails in~$\Gamma$.

\begin{restatable}{observation}{obscorridor}
    \label{obs:one_corridor_per_inner_segment}
    Every inner edge-segment of a drawing 
    is interior to exactly one trail.
\end{restatable}

A drawing is \emph{filled} if any two vertices~$u \neq v$  on the boundary of a cell~$c$ are joined by an uncrossed edge along~$\partial c$.
A $3$-plane, non-homotopic, connected, filled drawing of a graph on at least three vertices is \emph{$3$-saturated}.

\section{Crossing-Number and Edge-Density via Density Formula}
\label{sec:overview}
 
To obtain our upper bounds we prove a number of (in)equalities, each relating the number of certain cells, configurations, edges and crossings.
The Density Formula is one such equality.
In total, we obtain a system of linear inequalities where each quantity (such as $\abs{E}$, $(\abs{V}-2)$, $\abs{X}$, $\abs{\calC_2}$, $\abs{E_1}$, $\CIIIvO$, $\Clarge$, etc.) can be considered as a variable.
Setting the ``variable'' $(\abs{V}-2)$ to $1$, we can maximize the value of $\abs{X}$ by solving the obtained linear program~(LP).
The resulting maximum represents the number of crossings per vertex; more precisely, per $(\abs{V}-2)$.
We want to prove that the number of crossings in any $3$-plane drawing on $n$~vertices is at most~$5.5(n-2)$.
It thus suffices to show that the maximum value of $\abs{X}$ in the 
LP is $5.5$ if we set the variable representing the number of vertices to~$1$.
Our LP comprises~$21$~constraints, which are summarized in \cref{fig:certificates} in the appendix, and whose validity is proven in \cref{sec:constraints}.
Summing up all constraints with the coefficients in \cref{fig:certificates}, we obtain $\abs{X} \leq 5.5(\abs{V}-2)$.

If we maximize $\abs{E}$ instead, we obtain $\abs{E} \leq 5.5(\abs{V}-2)$ from the same constraints (with different coefficients; also in \cref{fig:certificates}).
Hence, by verifying that all $21$~constraints hold for every connected, non-homotopic $3$-plane drawing on $n \geq 3$ vertices, we obtain our result.

\main*

\section{Relating Crossing, Edge, Cell, Trail, and Configuration Counts}\label{sec:constraints}

In this section, we present a number of (in)equalities, each relating the number of certain cells, configurations, edges, or crossings. Due to space constraints we discuss only two of these inequalities, the rest can be found in \cref{app:constraints}. 
Our proof relies on the Density Formula for~$t=5$.
For this value of~$t$, $\Clarge$-cells contribute negatively in the formula. 
Intuitively, large cells account for many crossings: 
If many trails end in large cells, we obtain a lower bound on the sum~$\sum_{a \geq 6}a \abs{\calC_a}$ of sizes of large cells.
This yields a lower bound on the sum~$\sum_{c \in \calC_{\geq 6}}(\norm{c}-5)$ in the Density Formula, 
where $\calC_{\geq 6}$ denotes the set of large cells.
If there are few such trails, we obtain configurations that contain many crossed edges.

\begin{lemma}
    If $\Gamma$ is a $3$-saturated drawing, then
    \begin{align}
    \tag{5.A}
    \label{eq:5A}
    \sum_{a \geq 6}a \abs{\calC_a} \geq \cor{\CIVvI}{\Clarge} + \cor{\CVvI}{\Clarge} + \cor{\CIIIvO}{\Clarge} + \cor{\CVvO}{\Clarge} + 5\mul\CVIvII.
    \end{align}
\end{lemma}
\begin{proof}
    As we want to obtain a lower bound on the sum $\sum_{a \geq 6}a \abs{\calC_a}$, it suffices to count the number of vertex and edge-segment incidences of large cells. 
    Each trail that ends in a large cell enters this cell via an inner edge-segment.
    As no two trails share such an inner edge-segment, we obtain one edge-segment incidence for each such trail.

    A \CVIvII-cell is in particular large. 
    As it is incident to only one inner-segment, it is the endpoint of only one trail.
    We have not counted the remaining three edge-segment incidences and the two vertex incidences when considering trails.
    Therefore, each \CVIvII-cell yields at least five more edge-segment and vertex incidences.
\end{proof}

\begin{lemma}
    If~$\Gamma$ is a $3$-saturated drawing, then
    \begin{align}
    \tag{3.C}
    \label{eq:3C}
    \cor{\CIIIvO}{\CVvI} \leq \YTriQuadTri.
    \end{align}
\end{lemma}
\begin{proof}
    Consider a \CIIIvO-\CVvI-trail. 
    As every edge is crossed at most three times, the trail contains no $\CIVvO$-cell and we are in the situation represented in \cref{fig:3C}.
    The vertices~$u$ and~$v$ lie on the boundary of a cell~$c$. 
    As the drawing is $3$-saturated, the edge~$uv$ is contained in~$G$ and the cell~$c$ is a $\CVvII$-cell.
    The trail together with~$c$ forms a \YTriQuadTri-configuration.
    As every \CIIIvO-\CVvI-trail is only part of one such configuration, the statement follows.
    \begin{figure}
        \centering
        \includegraphics[page=3]{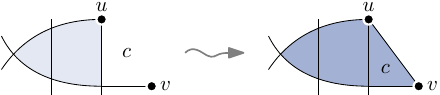}
        \caption{A \CIIIvO-\CVvI-trail (light blue). 
        It forms a \YTriQuadTri-configuration (dark) with an adjacent cell.}
        \label{fig:3C}
    \end{figure}
\end{proof}

\section{Discussion}
The \emph{$k$-planar crossing number} $\mathrm{cr}_k(G)$ is similar to the crossing number, except that the minimum is taken over all $k$-plane drawings of~$G$.
Clearly, $\mathrm{cr}(G) \leq \mathrm{cr}_k(G)$ for all~$k$ and~$G$. But there are $k$-planar $n$-vertex graphs $G$ with $\mathrm{cr}(G) \in \mathcal{O}(k)$ and $\mathrm{cr}_k(G) \in \Omega(kn)$ \cite[Theorem~2]{chimani2024crossingNumberRevisited}.
By \cref{thm:3-planar_upper_bound_crossings_and_edges}, every $3$-plane drawing of an $n$-vertex graph~$G$ has $\abs{X} \leq 5.5(n-2)$~crossings, and hence $\mathrm{cr}(G) \leq \mathrm{cr}_3(G) \leq 5.5(n-2)$. 
Although \cref{thm:3-planar_upper_bound_crossings_and_edges} is tight, 
we could have~$\mathrm{cr}(G),\mathrm{cr}_3(G)<5.5(n-2)$,
and a similar question arises for $2$-planar graphs.

\begin{question}{\ }
    \begin{itemize}
        \item Are there $3$-planar $n$-vertex graphs~$G$ with $\mathrm{cr}_3(G) = 5.5(n-2)$ or $\mathrm{cr}(G) = 5.5(n-2)$?
        \item Are there $2$-planar $n$-vertex graphs~$G$ with $\mathrm{cr}_2(G) = 3.\overline{3}(n-2)$ or $\mathrm{cr}(G) = 3.\overline{3}(n-2)$?
    \end{itemize}
\end{question}

\bibliographystyle{plainurl}
\bibliography{bibliography}

\appendix 

\section{Relating Crossing, Edge, Cell, Trail, and Configuration Counts}\label{app:constraints}

As explained above, we shall present a number of (in)equalities, each relating the number of certain cells, configurations, edges, or crossings.
We begin with three obvious equations.

\begin{observation}
    If~$\Gamma$ is a $3$-plane drawing, we have
    \begin{align}
        \tag{9.A} \label{eq:9A} \abs{E} &= \abs{E_{\times}} + \abs{E_0} \\
        \tag{8.A} \label{eq:8A} \abs{E_{\times}} &= \abs{E_1} + \abs{E_2} + \abs{E_3}  \\
        \tag{8.B} \label{eq:8B} 2\abs{X} &= \abs{E_1} + 2 \abs{E_2} + 3 \abs{E_3}  
    \end{align}
\end{observation}

Configurations are arrangements of several cells and edge-segments in a drawing~$\Gamma$ of a graph~$G$.
Formally, a \emph{configuration} is a connected labeled embedded subgraph~$H$ of the planarization~$\Lambda$ of~$\Gamma$. 
The label of a vertex~$v \in V(H)$ indicates whether~$v$ arises from a vertex of~$G$ or a crossing.
Two configurations are of the same \emph{type} if they are isomorphic as labeled embedded graphs (after possibly reflecting one of them).
We also denote the types by small pictograms such as $\YTriTriTri$ and $\YTriTri$, see \cref{fig:configurations_trails} (left) for an example.

A cell~$c$ is \emph{interior} to a configuration~$C$ if all of its boundary~$\partial c$ is part of~$C$.
We color interior cells gray in the pictograms.
An edge-segment is \emph{interior} to~$C$ if both its incident cells are interior to~$C$.
For example, a $\YTriTri$-configuration~$C$ has an interior $\CVvII$-cell and an interior $\CIVvI$-cell, whose shared outer edge-segment is the only interior edge-segment of~$C$.

Now, trails are specific configurations.
Let~$T = (A_1, \dots, A_{\ell})$ be a sequence of $\ell \geq 2$ cells whose dual is a walk\footnote{in~$3$-plane drawings these are always paths} in the planarization of~$\Gamma$.
The configuration~$T$ is a \emph{trail} if 
\textbf{(1)} neither $A_1$ nor $A_{\ell}$ is a \CIVvO-cell, 
\textbf{(2)} for each $i \in [\ell-1]$ cells $A_i$ and $A_{i+1}$ share an inner edge-segment (that is interior to~$T$), and
\textbf{(3)} each of $A_2,\ldots,A_{\ell-1}$ is a \CIVvO-cell $c$ whose two interior edge-segments are opposite on its boundary~$\partial c$.
See \cref{fig:configurations_trails} (right) for an example.
Note that only two of the four inner edge-segments of each \CIVvO-cell of~$T$ are interior to~$T$ (as no edge self-intersects).
Every trail~$T$ has two \emph{bounding edges}~$e_1$ and~$e_2$, which are those crossing all interior edge-segments of~$T$. 
An $A$-$B$-trail is a trail whose endpoints are cells of type~$A$ and~$B$ respectively. 
We denote by $\cor{A}{B}$ the number of $A$-$B$-trails in~$\Gamma$.
Note in particular that $\cor{A}{B} = \cor{B}{A}$.

Recall that a $3$-plane, non-homotopic, connected, filled drawing~$\Gamma$ of a graph on at least three vertices is called \emph{$3$-saturated}.
In every $3$-saturated drawing, all cells of size at most five are non-degenerate and of one of the following types:
$\CIIIvO$, $\CIVvO$, $\CIVvI$, $\CVvO$, $\CVvI$, and $\CVvII$; cf.~\cref{fig:cell-characterization}.

\begin{lemma}
\label{lem:augmentation_to_3-saturated-drawing}
    Every $3$-plane drawing~$\Gamma$ of a connected graph~$G$ on at least three vertices can be completed to a $3$-saturated drawing by only adding edges.
\end{lemma}
\begin{proof}
    Note that the drawing~$\Gamma$ is connected as~$G$ is connected.
    We therefore only need to enforce the filled property, while preserving all other properties.
    Suppose there is a cell~$c$ with vertices~$u$ and $v$ on its boundary, but no edge~$uv$ lies on~$\partial c$. 
    We can insert an edge~$uv$ within the cell~$c$, without creating new crossings. 
    The edge might be a parallel edge. 
    Yet, as no edge connecting $u$ and~$v$ lies on the boundary~$\partial c$, we did not create an empty lens.
    Therefore, the drawing is still non-homotopic.
    Neither does this create self-crossings or incident crossing edges.
    Inductively, the claim follows.
\end{proof}

\subsection{Lower Bounds on the Number of Cells and Configurations}

\begin{lemma}
    If~$\Gamma$ is a $3$-saturated drawing, then 
    \begin{align}
        \tag{2.A} \label{eq:2A} \CIVvI &= \cor{\CIVvI}{\CVvI} + \cor{\CIVvI}{\CVvO} + \cor{\CIVvI}{\Clarge} \\
        \tag{2.B} \label{eq:2B} 2\mul \CVvI &= \cor{\CIVvI}{\CVvI} + 2\cor{\CVvI}{\CVvI} + \cor{\CVvI}{\CIIIvO} + \cor{\CVvI}{\CVvO} + \cor{\CVvI}{\Clarge} \\
        \tag{2.C} \label{eq:2C} 3\mul\CIIIvO &= \cor{\CIIIvO}{\CVvI} + \cor{\CIIIvO}{\CVvO} + \cor{\CIIIvO}{\Clarge} \\
        \tag{2.D} \label{eq:2D} 5\mul\CVvO &= \cor{\CIIIvO}{\CVvO} + \cor{\CVvO}{\CIVvI} + \cor{\CVvO}{\CVvI} + 2\cor{\CVvO}{\CVvO} + \cor{\CVvO}{\Clarge}
    \end{align}
\end{lemma}
\begin{proof}
    We only prove the first equality. 
    The remaining equations can be shown in a similar way.
    Consider a \CIVvI-cell~$c$. 
    By \cref{obs:one_corridor_per_inner_segment}, the inner edge-segment on the boundary of~$c$ is interior to a trail starting in~$c$. 
    Let~$c'$ be the cell corresponding to the other end of the trail. 
    As parallel edges do not form an empty lens, $c'$ cannot be a \CIVvI-cell.
    Further, no two adjacent edges cross.
    Thus, $c'$ cannot be a \CIIIvO-cell.
    Recall that no end of a trail corresponds to a \CIVvO-cell, yet the cell~$c'$ has to be incident to an inner edge-segment. 
    Thus, $c'$ is a \CVvI-, a \CVvO- or a \Clarge-cell (i.e. a cell of size at least~$6$).
    As every \CIVvI-cell is only incident to one inner edge-segment, we obtain one trail for each such cell.
    Double-counting yields \eqref{eq:2A}.
\end{proof}

\begin{lemma}
    If~$\Gamma$ is a $3$-saturated drawing, then
    \begin{align}
    \tag{3.A}
    \label{eq:3A}
    \cor{\CIIIvO}{\CVvO} \leq \YTriQuad.
    \end{align}
\end{lemma}
\begin{proof}
    Consider a \CIIIvO-\CVvO-trail~$T$, see \cref{fig:3A}. 
    We use the same notation as in \cref{fig:3A}.
    Note that~$T$ contains no \CIVvO-cell as each of its bounding edges~$u'u$ and~$v'v$ is already crossed three times.
    For the same reason, one of the endpoints of $uu'$ and $vv'$ lies on the boundary of~$c$ and~$c'$, respectively, i.e. $u \in \partial c$ and $v \in \partial c'$.
    As the edge~$e$ is crossed at least twice, an endpoint~$w$ of~$e$ is incident to~$c$ or~$c'$. 
    We may assume without loss of generality that~$w$ lies on~$\partial c'$. 
    Observe that~$w$ and~$v$ are both incident to~$c'$.
    Since the drawing is $3$-saturated, $vw \in E(G)$ and $c'$ is a $\CVIvII$-cell.
    Similarly, we see that $uv \in E(G)$ and the cell~$c''$ is a \CVvII-cell.
    Thus, the cells~$c'$ and~$c''$ form a \YTriQuad-configuration~$C$.
    
    We now observe that no other \CIIIvO-\CVvO-trail is adjacent to~$C$.
    The two outer segments of the \CVvII-cell in~$C$ are part of the two bounding edges of the \CIIIvO-\CVvO-trail~$T$, see \cref{fig:3A} (right).
    As no two \CIIIvO-\CVvO-trails have the same bounding edges, each is adjacent to a different \YTriQuad-configuration and the inequality above follows.
    \begin{figure}[bt]
        \centering
        \includegraphics[page=1]{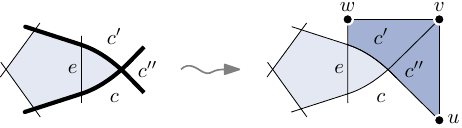}
        \caption{A $\CIIIvO$-$\CVvO$-trail (light blue) and its boundary (thick) is represented on the left. Such a trail is always adjacent to a $\YTriQuad$ (dark blue). The adjacency is represented on the right.}
        \label{fig:3A}
    \end{figure}
\end{proof}

\begin{lemma}
    If~$\Gamma$ is a $3$-saturated drawing, then
    \begin{align}
    \tag{3.B}
    \label{eq:3B}
    \cor{\CVvO}{\CVvI} \leq \YQuintQuadTri.
    \end{align}
\end{lemma}
\begin{proof}
    Consider a \CVvO-\CVvI-trail.
    As the drawing is $3$-plane, the trail contains no \CIVvO-cell, and we are in the situation represented in \cref{fig:3B}. 
    The edge~$e$ (represented by a thick line) is crossed three times.
    Thus, one of its endpoints lies on the boundary of the cell~$c$. 
    Let~$u$ denote this vertex.
    Note that the vertex~$v$, which is the only vertex on the \CVvI-cell, is also incident to~$c$.
    As the drawing is $3$-saturated, $uv$ is an edge and $c$ is a \CVvII-cell.
    The \CVvO-\CVvI-trail together with~$c$ forms a \YQuintQuadTri.
    For every \CVvO-\CVvI-trail, we obtain a distinct \YQuintQuadTri (which may share a \CVvII-cell with another such configuration) and the inequality follows. 
    \begin{figure}[tb]
        \centering
        \includegraphics[page=2]{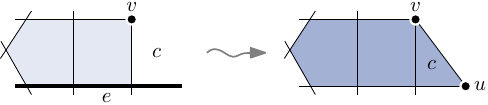}
        \caption{A \CVvO-\CVvI-trail (light blue). The bounding edge~$e$ (thick) is crossed three times. Such a trail forms a \YQuintQuadTri with an adjacent cell. The \YQuintQuadTri (dark blue) is represented on the right.}
        \label{fig:3B}
    \end{figure}
\end{proof}

\begin{lemma}
     If~$\Gamma$ is a $3$-saturated drawing, then
    \begin{align}
    \tag{3.D}
    \label{eq:3D}
    \CIVvI \leq \YTriTri + \YTriTriTri.
    \end{align}
\end{lemma}
\begin{proof}
    We show that each \CIVvI-cell is contained in a \YTriTri- or a \YTriTriTri-configuration.
    Consider the inner edge-segment of a \CIVvI-cell~$c$ and let~$e$ be the corresponding edge.
    We denote the two cells sharing an outer edge-segment with~$c$ by $c'$ and~$c''$.
    As the edge~$e$ is crossed at most three times, an endpoint~$v$ of~$e$ lies on the boundary of~$c'$ or~$c''$. 
    Without loss of generality, we are therefore in the situation represented in \cref{fig:3D}.
    The two vertices~$u$ and~$v$ lie on the boundary of the cell~$c''$.
    As~$\Gamma$ is $3$-saturated, $uv$ is an edge of~$G$ and~$c''$ is a \CVvII-cell. 
    
    If~$e$ is crossed three times, the three cells $c', c$ and $c''$ are in \YTriTri-configuration.
    Otherwise, $e$ is crossed twice and the same argument as above shows that $c'$ is also a \CVvII-cell. 
    In particular, the cells $c', c$ and $c''$ are in \YTriTriTri-configuration.

    As every \CIVvI-cell is only contained in one such configuration, the upper bound on the number of \CIVvI-cells follows.
    \begin{figure}[tb]
        \centering
        \includegraphics[page=4]{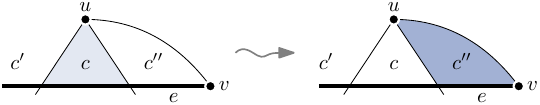}
        \caption{A \CIVvI-cell~$c$ (light blue). 
        One of the two cells sharing an outer edge-segment with~$c$ is a \CVvII-cell (dark blue).}
        \label{fig:3D}
    \end{figure}
\end{proof}

\begin{lemma}
    If~$\Gamma$ is a $3$-saturated drawing, then
    \begin{align}
    \tag{3.E}
    \label{eq:3E}
    \cor{\CIVvI}{\CVvI} \leq \frac{1}{2} \abs{E_1} + \YTriQuadQuad.
    \end{align}
\end{lemma}
\begin{proof}
    Note that it suffices to show that every $\CIVvI$-$\CVvI$-trail forms a $\YTriQuadQuad$-configuration or has a bounding edge that is crossed at most once.
    The claim then follows as every edge is a bounding edge of at most two trails.

    Consider a \CIVvI-\CVvI-trail. 
    As every edge is crossed at most three times, the trail contains at most one~$\CIVvO$-cell.
    
    If the trail contains no such cell, we are in the situation depicted on the left of \cref{fig:3E}. 
    Note that the bounding edge~$e$ of the trail is crossed only once. 
    
    If the trail contains a $\CIVvO$-cell, we are in the situation represented on the right of \cref{fig:3E}. 
    The trail forms a $\YTriQuadQuad$-configuration.

    Thus, the number of \CIVvI-\CVvI-trails yields a lower bound on the sum of half the number of edges that are crossed only once, and the number of $\YTriQuadQuad$-configurations.
    \begin{figure}[t]
        \centering
        \includegraphics[page=5]{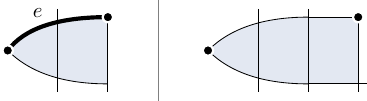}
        \caption{A \CIVvI-\CVvI-trail with no $\CIVvO$-cell is depicted on the left, and with one $\CIVvO$-cell on the right.}
        \label{fig:3E}
    \end{figure}
\end{proof}

\begin{lemma}
    If~$\Gamma$ is a $3$-saturated drawing, then
    \begin{align}
    \tag{4.A}
    \label{eq:4A}
    2\mul\cor{\CVvO}{\CVvO} + \cor{\CIVvI}{\CVvO} + \cor{\CIIIvO}{\CVvO} - 4\mul \CVvO \leq \YStar.
    \end{align}
\end{lemma}
\begin{proof}
    Let~$T$ be a trail ending in a $\CVvO$-cell and let~$e_1$ and~$e_2$ be the two edges on its boundary.
    If the other end of~$T$ corresponds to a $\CIVvI$-, $\CVvO$- and $\CIIIvO$-cell, the edges~$e_1$ and~$e_2$ have the same number of crossings within the trail.
    We therefore call these cells \emph{crossing-even} and use the same term for trails where one endpoint corresponds to a $\CVvO$-cell, the other to a crossing-even cell. 
    A $\CVvO$-cell where every incident trail ends in a crossing-even cell is \emph{saturated}.
    We denote the number of saturated $\CVvO$-cells by~$\ell$.

    \proofsubparagraph{An upper bound on the number~$\bm{\ell}$ of saturated $\CVvO$-cells.}
    We first show that every saturated $\CVvO$-cell is part of at least one $\YStar$-configuration. 
    As no such configuration contains two $\CVvO$-cells, we then obtain
    \[
    \label{eq:num_saturated_c5v0_at_most_number_ystar}
    \ell \leq \YStar. \tag{$\star$}
    \]

    Consider a saturated $\CVvO$-cell~$c$. 
    We say that a trail incident to~$c$ is \emph{uncrossed} if its boundary contains no crossing that is not incident to~$c$.
    Otherwise, we say that it is \emph{crossed}.
    In fact, all such uncrossed trails end in a $\CIVvI$-cell and contain no $\CIVvO$-cell. 
    Every trail~$T$ incident to~$c$ contains an inner edge-segment~$s(T)$ on the boundary of~$c$ in its interior.
    Two such trails~$T$ and~$T'$ are called \emph{consecutive} if the edge-segments~$s(C)$ and~$s(C')$ share a crossing.
    
    We first show that there are at least three consecutive uncrossed trails incident to the saturated cell~$c$. 
    If none of the incident trails of~$c$ is crossed, the claim clearly holds.
    Otherwise, some trail~$A$ is crossed, see \cref{fig:4_oversaturation} (left).
    Recall that~$A$ is crossing-even.
    Thus, the two edges on its boundary are crossed three times.
    Therefore, the trails~$C$ and $D$ are uncrossed.
    The same argument shows that only one of the two remaining trails~$B$ and~$E$ is crossed.
    We therefore obtain three consecutive uncrossed trails.
    
    As every uncrossed trail ends in a $\CIVvI$-cell and contains no $\CIVvO$-cell, we are in the situation depicted in \cref{fig:4_oversaturation} (right).
    \begin{figure}[tb]
        \centering
        \includegraphics[page=6]{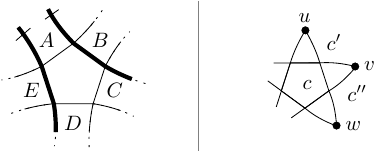}
        \caption{Left: A saturated $\CVvO$-cell with its incident trails is represented. One of the trails is crossed. The corresponding boundary edges are represented with thick lines.
        Right: A saturated $\CVvO$-cell with three consecutive uncrossed trails.}
        \label{fig:4_oversaturation}
    \end{figure}
    Note that the vertices~$u$ and~$v$ lie on the boundary of a cell~$c'$. 
    Thus, $uv \in E(G)$ and~$c'$ is a $\CVvII$-cell as~$\Gamma$ is $3$-saturated.
    Similarly, we see that $vw \in E(G)$ and~$c''$ is a $\CVvII$-cell.
    The cells~$c$, $c'$ and~$c''$ form a $\YStar$-configuration.
    It follows that every saturated $\CVvO$-cell is part of a $\YStar$-configuration.

    \medskip
    
    \proofsubparagraph{A lower bound on the number~$\bm{\ell}$ of saturated $\CVvO$-cells.}
    We define \[k \coloneqq 2\mul\cor{\CVvO}{\CVvO} + \cor{\CIVvI}{\CVvO} + \cor{\CIIIvO}{\CVvO},\]
    i.e. $k$ corresponds to the number of incidences between crossing-even trails and~$\CVvO$-cells.
    It remains to show that 
    \[
    \label{eq:num_saturated_c5v0_lower_bounded_by_crossing_even_ccorridors}
    k  - 4\mul \CVvO \leq \ell. \tag{$\star\star$}
    \]
    The claim then follows by combining~\eqref{eq:num_saturated_c5v0_at_most_number_ystar} and \eqref{eq:num_saturated_c5v0_lower_bounded_by_crossing_even_ccorridors}.
    
    Every $\CVvO$-cell that is not saturated is incident to at most four crossing-even trails while every saturated $\CVvO$-cell is incident to five such trails.
    As there are $k$~incidences between crossing-even trails and $\CVvO$-cells, we obtain $k \leq 4(\CVvO-\ell) + 5\ell = 4\CVvO + \ell$.
    Thus, \eqref{eq:num_saturated_c5v0_lower_bounded_by_crossing_even_ccorridors} holds.
\end{proof}

\begin{lemma}
    If~$\Gamma$ is a $3$-saturated drawing, then
    \begin{align}
    \tag{4.B}
    \label{eq:4B}
    \YTriTriTri \geq \YStar.
    \end{align}
\end{lemma}
\begin{proof}
    Each \YStar-configuration contains a \YTriTriTri.
    As no two \YStar-configurations share a \YTriTriTri, the claim follows.
\end{proof}

\begin{lemma}
    For every drawing~$\Gamma$, we have
    \begin{align}
    \tag{9.C}
    \label{eq:9C}
        2\mul\CVvII \geq  \YQuintQuadTri +  \YTriQuad +  \YTriQuadTri +  \YTriTri + 2\YTriTriTri.
    \end{align}
\end{lemma}
\begin{proof}
    We give a lower bound on the number of \CVvII-cells by double-counting the number of outer edge-segments incident to a \CVvII-cell.

    Every \CVvII-cell is incident to two outer edge-segments.

    Recall that an edge-segment is interior to a configuration~$C$ if both its incident cells are interior to~$C$.
    In each of the configurations $\YQuintQuadTri$, $\YTriQuad$, $\YTriQuadTri$, $\YTriTri$ and $\YTriTriTri$, exactly one outer edge-segment~$s_c$ of each interior $\CVvII$-cell~$c$ is interior to the configuration.
    Note in particular that we did not count any such edge-segment twice due to the choice of the configurations we consider.
    Indeed, any two configurations may share a \CVvII-cell~$c$, but the corresponding outer edge-segments~$s_c$ do not coincide.
    As each $\YQuintQuadTri$, $\YTriQuad$, $\YTriQuadTri$ and $\YTriTri$ contains one outer edge-segment that is incident to a \CVvII-cell in its interior, and \YTriTriTri contains two such segments, the lower bound on the number of \CVvII-cells follows.
\end{proof}

\subsection{Lower Bounds on the Number of Edges}

\begin{lemma}
    For every drawing $\Gamma$, we have
    \begin{align}
    \tag{6}
    \label{eq:6}
        4 \abs{E_{\times}} \geq 2\mul\CIVvI + 2\mul\CVvI + 2\mul \CVvII + 2\mul \CVIvII.
    \end{align}
\end{lemma}
\begin{proof}
    We double count the number~$\ell$ of incidences between outer edge-segments and cells.

    Every crossed edge has exactly two outer edge-segments. 
    Each of them is incident to two cells.
    Thus, $4 \abs{E_{\times}} = \ell$. 

    Every \CIVvI-cell, every \CVvI-cell, every \CVvII-cell and every \CVIvII-cell is incident to exactly two outer edge-segments.
    We therefore have, $\ell \geq 2\CIVvI + 2\mul\CVvI + 2\mul \CVvII + 2\mul \CVIvII$ which yields the claim.
\end{proof}

\begin{lemma}
    For every drawing~$\Gamma$, we have
    \begin{align}
        2\abs{E_2} + 4 \abs{E_3} \geq 3\mul\CIIIvO &+  \CIVvI + 4\mul\CIVvO + 2\mul \CVvI + 5\mul\CVvO \notag \\
        &+ \cor{\CIVvI}{\Clarge} + \cor{\CVvI}{\Clarge} + \cor{\CIIIvO}{\Clarge} + \cor{\CVvO}{\Clarge}. \tag{7}
    \label{eq:7}
    \end{align}
\end{lemma}
\begin{proof}
    We double count the number~$\ell$ of incidences between inner edge-segments and cells.

    An edge that is crossed $k$ times is split into $k+1$ edge-segments, $k-1$ of which are inner edge-segments. 
    Each inner edge-segment is incident to two cells.
    We therefore have $2 \abs{E_2} + 4 \abs{E_3} \geq \ell$.

    On the other hand, every \CIIIvO-cell, every \CIVvI-cell, every \mul\CIVvO-cell, every \CVvI-cell and every \CVvO-cell is incident to $3$, $1$, $4$, $2$ and $5$ inner edge-segments respectively.
    In order to obtain a lower bound for the number of inner edge-segment incidences with large cells, it suffices to count the number of trails ending in large cells.
    Indeed, each such trail enters a large cell through an inner edge-segment.
    This yields the right side of the inequality above.
\end{proof}

\begin{lemma}
    For every drawing~$\Gamma$, we have
    \begin{align}
    \tag{8.C}
    \label{eq:8C}
        \abs{E_2} \geq \frac{1}{2} \mul \YTriQuadQuad + \YStar
    \end{align}
\end{lemma}
\begin{proof}
    Every \YTriQuadQuad-configuration contains an edge that is crossed twice. 
    This edge may be contained in at most two such configurations. 
    Similarly, every \YStar-configuration contains an edge that is crossed twice. 
    Note that this edge cannot be contained in any other \YStar- or \YTriQuadQuad-configuration.
    
    Thus, the sum of $\frac{1}{2}\YTriQuadQuad$ and $\YStar$ provides a lower bound on the number~$\abs{E_2}$ of edges that are crossed exactly twice.
\end{proof}

Double-counting the number of incidences of uncrossed edges with cells yields the following.
\begin{lemma}
    For every drawing~$\Gamma$, we have
    \begin{align}
    \tag{9.B}
    \label{eq:9B}
        2\abs{E_0} \geq \CVvII + \CVIvII.
    \end{align}
\end{lemma}

\subsection{Relating the Number of Cells, Crossings and Edges}

\begin{corollary}
    If $\Gamma$ is a $3$-saturated drawing on at least three vertices, then
    \begin{align}
    \tag{5.B}
    \label{eq:5B}
    \abs{X} \leq 5\abs{V} + 2\mul\CIIIvO + \CIVvO + \CIVvI -\frac{1}{6}\sum_{a \geq 6} a \abs{\calC_a} - \abs{E}.
    \end{align}
\end{corollary}
\begin{proof}
    Recall that every $3$-saturated drawing is in particular connected.
    As the graph contains at least three vertices, every cell has size at least~$3$. 
    Considering $t=5$ in the density formula (\cref{thm:density_formula}) yields
    \begin{equation}
        \label{eq:density_first_inequality}
        \abs{E} \leq 5\abs{V} - \sum_{c \in \calC}\left(\norm{c} - 5 \right) - \abs{X} \tag{$\star$}.
    \end{equation}
    Thus, only cells of size at most~$4$ have a positive contribution to the right side.
    Recall that there is only one cell-type of size~$3$, namely \CIIIvO, and only two types of size~$4$: \CIVvO \, and \CIVvI. 
    Cells of size~$5$ have no contribution. 
    For $a \geq 6$, we have $a-5 \geq \frac{1}{6}a$. 
    We therefore obtain
    \[
    \sum_{c \in C, \norm{c} \geq 6}\left(\norm{c} - 5\right) = \sum_{a \geq 6} \abs{\calC_a}(a-5) \geq \frac{1}{6} \sum_{a \geq 6} a \abs{\calC_a}.
    \]
    Together with \eqref{eq:density_first_inequality}, we get
    \[
        \abs{E} \leq 5\abs{V} + 2\mul\CIIIvO + \CIVvO + \CIVvI - \frac{1}{6} \sum_{a \geq 6} a \abs{\calC_a} - \abs{X}. \qedhere
    \]    
\end{proof}


    \begin{figure}[pt]
        \small
\setlength{\abovedisplayskip}{0pt}
\setlength{\belowdisplayskip}{0pt}
    \begin{align*}
        &&\text{(In)equality} &&& \text{$\abs{E}$} &&\text{$\abs{X}$}
        \\
        \text{\eqref{eq:2A}} &&\cor{\CIVvI}{\CVvI} + \cor{\CIVvI}{\CVvO} + \cor{\CIVvI}{\Clarge} - \CIVvI &= 0 
        &&\coeff{\frac{-5}{16}}
        &&\coeff{\frac{-7}{16}}
        \\
        \text{\eqref{eq:2B}} &&\cor{\CIVvI}{\CVvI} + 2\cor{\CVvI}{\CVvI} + \cor{\CVvI}{\CIIIvO} + \cor{\CVvI}{\CVvO} + \cor{\CVvI}{\Clarge} - 2\mul \CVvI &= 0
        &&\coeff{\frac{5}{16}}
        &&\coeff{\frac{5}{16}}
        \\
        \text{\eqref{eq:2C}} &&\cor{\CIIIvO}{\CVvI} + \cor{\CIIIvO}{\CVvO} + \cor{\CIIIvO}{\Clarge} - 3\mul\CIIIvO &= 0
        &&\coeff{\frac{-11}{24}}
        &&\coeff{\frac{-11}{24}}
        \\
        \text{\eqref{eq:2D}} && \cor{\CIIIvO}{\CVvO} + \cor{\CVvO}{\CIVvI} + \cor{\CVvO}{\CVvI} + 2\cor{\CVvO}{\CVvO} + \cor{\CVvO}{\Clarge} - 5\mul\CVvO &= 0
        &&\coeff{\frac{1}{8}}
        &&\coeff{\frac{-3}{8}}
        \\
        \text{\eqref{eq:3A}} &&\cor{\CIIIvO} {\CVvO} - \YTriQuad &\leq 0
        &&\coeff{\frac{7}{48}}
        &&\coeff{\frac{1}{48}}
        \\
        \text{\eqref{eq:3B}} &&\cor{\CVvO}{\CVvI} - \YQuintQuadTri &\leq 0
        &&\coeff{0}
        &&\coeff{\frac{1}{16}}
        \\
        \text{\eqref{eq:3C}} &&\cor{\CIIIvO}{\CVvI} - \YTriQuadTri &\leq 0
        &&\coeff{\frac{3}{16}}
        &&\coeff{\frac{7}{48}}
        \\
        \text{\eqref{eq:3D}}
        &&\CIVvI - \YTriTri - \YTriTriTri &\leq 0
        &&\coeff{\frac{3}{16}}
        &&\coeff{\frac{5}{16}}
        \\
        \text{\eqref{eq:3E}}
        &&2\cor{\CIVvI}{\CVvI} - \abs{E_1} - 2\YTriQuadQuad &\leq 0 
        &&\coeff{0}
        &&\coeff{\frac{1}{16}}
        \\
        \text{\eqref{eq:4A}}
        &&2\mul\cor{\CVvO}{\CVvO} + \cor{\CIVvI}{\CVvO} + \cor{\CIIIvO}{\CVvO} - 4\mul \CVvO - \YStar &\leq 0
        &&\coeff{\frac{3}{16}}
        &&\coeff{\frac{13}{16}}
        \\
        \text{\eqref{eq:4B}}
        &&\YStar - \YTriTriTri &\leq 0
        &&\coeff{\frac{3}{16}}
        &&\coeff{\frac{5}{16}}
        \\
        \text{\eqref{eq:5A}}
        &&\cor{\CIVvI}{\Clarge} + \cor{\CVvI}{\Clarge} + \cor{\CIIIvO}{\Clarge} + \cor{\CVvO}{\Clarge} + 5\mul\CVIvII - \sum_{a \geq 6}a \abs{\calC_a} &\leq 0
        &&\coeff{\frac{11}{60}}
        &&\coeff{\frac{11}{60}}
        \\
        \text{\eqref{eq:5B}}
        && \sum_{a \geq 6} a \abs{\calC_a} + 6\abs{E} + 6\abs{X} - 12\mul\CIIIvO - 6\CIVvO - 6\CIVvI  &\leq 30(\abs{V}-2)
        &&\coeff{\frac{11}{60}}
        &&\coeff{\frac{11}{60}}
        \\
        \text{\eqref{eq:6}}
        &&2\mul\CIVvI + 2\mul\CVvI + 2\mul \CVvII + 2\mul \CVIvII - 4 \abs{E_{\times}} &\leq 0
        &&\coeff{\frac{13}{80}}
        &&\coeff{\frac{3}{80}}
        \\
        \text{\eqref{eq:7}}
        &&\cor{\CIVvI}{\Clarge} + \cor{\CVvI}{\Clarge} + \cor{\CIIIvO}{\Clarge} + \cor{\CVvO}{\Clarge} \phantom{- 4 \abs{E_3} }\\
        \phantom{\text{\eqref{eq:7}}} &&
        \phantom{\cor{\CIVvI}{\Clarge}} + 3\mul\CIIIvO +  \CIVvI + 4\mul\CIVvO + 2\mul \CVvI + 5\mul\CVvO -2\abs{E_2} - 4 \abs{E_3}  &\leq 0
        &&\coeff{\frac{11}{40}}
        &&\coeff{\frac{11}{40}}
        \\
        \text{\eqref{eq:8A}} && \abs{E_1} + \abs{E_2} + \abs{E_3} - \abs{E_{\times}} &= 0
        &&\coeff{\frac{-11}{20}}
        &&\coeff{\frac{19}{20}}
        \\
        \text{\eqref{eq:8B}} &&\abs{E_1} + 2 \abs{E_2} + 3 \abs{E_3} - 2\abs{X} &= 0
        &&\coeff{\frac{11}{20}}
        &&\coeff{\frac{1}{20}}
        \\
        \text{\eqref{eq:8C}}
        && \YTriQuadQuad + 2\YStar - 2\abs{E_2} &\leq 0
        &&\coeff{0}
        &&\coeff{\frac{1}{4}}
        \\
        \text{\eqref{eq:9A}} &&\abs{E_{\times}} + \abs{E_0} - \abs{E} &= 0
        &&\coeff{\frac{1}{10}}
        &&\coeff{\frac{11}{10}}
        \\
        \text{\eqref{eq:9B}}
        &&\CVvII + \CVIvII - 2\abs{E_0} &\leq 0
        &&\coeff{\frac{1}{20}}
        &&\coeff{\frac{11}{20}}
        \\
        \text{\eqref{eq:9C}}
        && \YQuintQuadTri + \YTriQuad +  \YTriQuadTri + \YTriTri + 2\YTriTriTri - 2\mul\CVvII &\leq 0
        &&\coeff{\frac{3}{16}}
        &&\coeff{\frac{5}{16}}
    \end{align*}
    \caption{Certificates for the upper bound on the number of edges and crossings in $3$-saturated drawings in terms of the number of vertices. Each row corresponds to one inequality. 
In order to obtain the upper bound on the number of edges, we multiply each inequality with the third entry in the corresponding row and sum up all the inequalities.
To obtain the upper bound on the number of crossings we proceed likewise using the fourth entry of each row as a coefficient.}
    \label{fig:certificates}
\end{figure}

\end{document}